\documentclass[11pt]{amsart}
\usepackage{mathrsfs}
\usepackage{graphicx,diagrams,pstricks}
\usepackage{verbatim}
\usepackage{epsfig}
\usepackage{amssymb}
\usepackage{amsthm}
\usepackage[all,knot,arc]{xy}
\usepackage{color}
\usepackage{hyperref}
\hypersetup{colorlinks=true,citecolor=blue,linkcolor=blue}
\usepackage{amsfonts,amsmath}
\newtheorem{theorem}{Theorem}[section]
\newtheorem{lemma}[theorem]{Lemma}
\newtheorem{proposition}[theorem]{Proposition}
\newtheorem{corollary}[theorem]{Corollary}

 \usepackage{palatino}

\theoremstyle{definition}

\newtheorem{conj}[theorem]{Conjecture}
\newtheorem{question}[theorem]{Question}

\newcommand{\s}{\mathfrak{s}}

\newcommand{\spinc}{{\mbox{spin$^c$} }}

\newcommand{\zee}{\mathbb{Z}}
\newcommand{\arr}{\mathbb{R}}

\newcommand{\cee}{\mathbb{C}}

\newcommand{\wt}{\widetilde}

\newcommand{\A}{{\mathcal{A}}}

\usepackage{graphicx,amsfonts}

\newarrow{Dashto}{}{dash}{}{dash}>

\numberwithin{equation}{section}

\begin{document}

\title[Pseudoconvex embeddings of Brieskorn spheres]{Obstructing pseudoconvex embeddings and contractible Stein fillings for Brieskorn spheres}
\author{Thomas E. Mark}
\address{Department of Mathematics, University of Virginia}
\email{tmark@virginia.edu	}
\author{B\"ulent Tosun}
\address{Department of Mathematics, University of Alabama}
\email{btosun@ua.edu	}

\begin{abstract} A conjecture due to Gompf asserts that no nontrivial Brieskorn homology sphere admits a pseudoconvex embedding in $\cee^2$, with either orientation. A related question asks whether every compact contractible 4-manifold admits the structure of a Stein domain. We verify Gompf's conjecture, with one orientation, for a family of Brieskorn spheres of which some are known to admit a smooth embedding in $\cee^2$. With the other orientation our methods do not resolve the question, but do give rise to an example of a contractible, boundary-irreducible 4-manifold that admits no Stein structure with either orientation, though its boundary has Stein fillings with both orientations.
\end{abstract}

\maketitle

\section{Introduction}
A Stein manifold is a complex manifold admitting a proper Morse function that is bounded below and strictly plurisubharmonic. By a compact Stein manifold, or {\it Stein domain}, we mean a sub-level set corresponding to a regular value of such a function. Smooth manifolds supporting a Stein structure were characterized in terms of their handle structure by Eliashberg \cite{EliashbergStein}, with some refinements in the case of Stein surfaces due to Gompf \cite{Gompf}. Since that time, Stein structures have been shown to exhibit a remarkable degree of flexibility in certain senses, yet their geometric structure also imposes enough rigidity to allow classification results in other cases. An instance of flexibility is given by Gompf in \cite{gompf13}, in which it is shown that a codimension-0 submanifold $U$ of a complex surface is isotopic to a Stein subsurface if and only if the induced complex structure is homotopic through almost-complex structures to a Stein structure on $U$. Using this result, Gompf exhibits a variety of interesting examples of Stein manifolds in $\cee^2$ including domains of holomorphy that are diffeomorphic to non-standard smooth structures on $\arr^4$, and contractible Stein domains with non-simply-connected boundary.

In particular, Gompf finds many examples of hyperbolic integer homology 3-spheres that embed in $\cee^2$ as the boundary of a Stein domain in $\cee^2$. We refer to such an embedding of a 3-manifold as a {\it pseudoconvex embedding}. On the other hand, Gompf makes the following bold conjecture:

\begin{conj}\cite{gompf13}\label{conj} No Brieskorn integer homology sphere $\Sigma$ (other than $S^3$) admits a pseudoconvex embedding in $\cee^2$, with either orientation. 
\end{conj}

As Gompf points out, many Brieskorn spheres do not embed even smoothly in $\cee^2$. This paper provides the first evidence for Conjecture \ref{conj} in some cases where $\Sigma$ does have a smooth embedding in $\cee^2$.

A question closely related to the existence of pseudoconvex embeddings has been raised in contact- and symplectic-geometric circles:

\begin{question}\label{steinquest} Does every compact contractible 4-manifold with Stein fillable boundary admit a Stein structure?
\end{question}

See, for example, Yasui \cite{yasui17}, where the connection between this question---specialized to the case that the boundary is the 3-sphere---and the smooth 4-dimensional Poincar\'e conjecture is described. The results below show that the answer to Question \ref{steinquest} is negative in general.

A relation between Question \ref{steinquest} and Conjecture \ref{conj} is as follows. First note that if a contractible 4-manifold $X$ admits a handle decomposition with a single 1-handle, a single 2-handle, and no 3- or 4-handles, then $X$ embeds smoothly in $\cee^2$ (the same is true if there are more 1- and 2-handles, as long as the corresponding presentation of the fundamental group is Andrews-Curtis trivial; see \cite[Example 3.2]{gompf13}). If such a contractible manifold admits a Stein structure, then using Gompf's isotopy result above, it follows that $X$ is diffeomorphic to a Stein domain in $\cee^2$ and in particular its boundary admits a pseudoconvex embedding in $\cee^2$. Conversely, if a homology 3-sphere has a pseudoconvex embedding in $\cee^2$, then it is clearly the boundary of an acyclic Stein manifold (for the orientation question, see \cite[Proposition 7.4]{gompf13}).

We consider these issues in the context of the family of Brieskorn spheres $\Sigma(2,3, 6m\pm 1)$. For many members of this ubiquitous family, Gompf's conjecture \ref{conj} can be answered affirmatively by purely topological means. We have:
\begin{itemize}
\item For odd $m$, both $\Sigma(2,3,6m-1)$ and $\Sigma(2,3, 6m+1)$ have nontrivial Rohlin invariant, hence neither bounds an acyclic 4-manifold. In particular these manifolds do not admit even a smooth embedding in $\cee^2$.
\item For even $m$, the Brieskorn sphere $\Sigma(2,3,6m-1)$ has $R = 1$, where $R$ is the invariant of Fintushel-Stern. By  Theorem 1.1 of \cite{FSpseudofree}, it follows that none of these manifolds bound acyclic either.
\end{itemize}

This leaves the family $\Sigma(2,3,12n+1)$, for which no standard invariants (e.g., Rohlin or $\bar{\mu}$, Fintushel-Stern's $R$, the Heegaard-Floer $d$-invariant, or Manolescu's lift $\beta$ of the Rohlin invariant \cite{manolescuPin2}) obstruct an acyclic filling. In fact it is known by work of Akbulut-Kirby \cite{AKmazur} and Casson-Harer \cite{cassonharer} in the case $n=1$, and Fickle \cite{fickle} when $n=2$, that both $\Sigma(2,3,13)$ and $\Sigma(2,3,25)$ are the boundaries of smooth contractible 4-manifolds of ``Mazur type,'' meaning each has a handle decomposition with one 1-handle and one 2-handle. In particular, both of these admit smooth embeddings in $\cee^2$. For $n\geq 3$, neither an acyclic 4-manifold bounding $\Sigma(2,3,12n+1)$ nor an embedding in $\cee^2$ appear to be known, but it seems plausible that both exist.

The main result of this paper is the following, where we write $-M$ for a manifold $M$ after reversing its orientation.

\begin{theorem}\label{mainthm} Suppose $X$ is a smooth compact oriented acyclic 4-manifold whose boundary is the Brieskorn sphere $-\Sigma(2,3,12n+1)$ with the opposite of its usual orientation, for some $n\geq 1$. Then $X$ admits no symplectic structure weakly filling a contact structure on its boundary.
\end{theorem}


Note that $\Sigma(2,3,12n+1)$, with either orientation, does admit Stein fillings (and with one orientation supports Stein fillable contact structures having $\theta = -2$: see below).
However, since a Stein structure provides a weak symplectic filling of the induced contact structure on the boundary, Theorem \ref{mainthm} together with the constructions of Akbulut-Kirby, Casson-Harer, and Fickle, proves the following, which gives a negative answer to Question \ref{steinquest}:
 
\begin{corollary}\label{nonsteincor} No acyclic oriented 4-manifold with boundary $-\Sigma(2,3,12n+1)$ admits a Stein structure. In particular, there exist contractible, boundary-irreducible 4-manifolds $X$ that do not admit Stein structures.
\end{corollary}

In fact, there exist $X$ as in Corollary \ref{nonsteincor} such that neither $X$ nor $-X$ admit a Stein structure, as follows from Theorem \ref{AKthm} below.

Correspondingly, relating to Conjecture \ref{conj}, we have:

\begin{corollary}\label{embcor} No member of the family $-\Sigma(2,3,12n+1)$ admits a pseudoconvex embedding in $\cee^2$.
\end{corollary}

Theorem \ref{mainthm} and its corollaries follow from results of the second author on classification of contact structures \cite{T16}, as we now demonstrate. Recall that a Stein domain induces a natural contact structure on its boundary as the field of complex tangencies. The homotopy class of an oriented tangent 2-plane field $\xi$ on an oriented homology sphere $Y$ is determined by an invariant $\theta(\xi)\in 2\zee$ defined by
\[
\theta(\xi) = c_1^2(X,J) -3\sigma(X) -2\chi(X),
\]
for any almost-complex 4-manifold $(X,J)$ with $\partial X = Y$ and such that $\xi$ is the field of complex tangencies $TY\cap J(TY)$. Here $\chi(X)$ is the Euler characteristic and $\sigma(X)$ is the signature of the cup product pairing on $H^2(X;\zee)$. Thus, if $(Y,\xi)$ is the contact boundary of an acyclic Stein manifold $(X,J)$, then necessarily $\theta(\xi) = -2$. The same conclusion holds if $(Y,\xi)$ is weakly symplectically filled by an acyclic manifold $(X,\omega)$, since we can select an almost-complex structure compatible with both $\omega$ and $\xi$.

\begin{theorem}\label{thm2}\cite{T16} For any $m\geq 1$, the oppositely-oriented Brieskorn manifold $-\Sigma(2,3,6m+1)$ admits exactly $\frac{1}{2}m(m+1)$ tight contact structures up to isotopy. All these contact structures are homotopic and have $\theta = 2$; at least $m$ of them are Stein fillable.
\end{theorem}

The proof, given by the second author in \cite{T16}, follows from an adaptation of a beautiful argument due to Ghiggini and Van Horn-Morris \cite{GVHM} who prove an analogous result for the manifolds $-\Sigma(2,3,6m-1)$ . Taking $m= 2n$, Theorem \ref{thm2} implies Theorem \ref{mainthm}, because any (weakly) fillable contact structure is tight, and since none of the tight contact structures have $\theta = -2$, none are filled by an acyclic symplectic manifold.


With its natural orientation, the classification of tight contact structures on $\Sigma(2,3,6m+1)$ is simpler:

\begin{theorem}\label{thm3} For any $m\geq 1$, the Brieskorn sphere $\Sigma(2,3,6m+1)$ admits exactly two tight contact structures up to isotopy. Both are Stein fillable, and both have $\theta = -2$; in fact the two contact structures are contactomorphic.
\end{theorem}

Thus, though the classification of contact structures is simple with this orientation, the homotopy obstruction does not rule out acyclic fillings this case. The proof of Theorem \ref{thm3} is given in Section \ref{contactsec}.

However, we have the following result whose proof appears in Section \ref{HFsec}.

\begin{theorem}\label{AKthm} Let $B$ denote the contractible smooth 4-manifold with $\partial B = \Sigma(2,3,13)$ constructed by Akbulut and Kirby \cite{AKmazur}, whose handle diagram appears in Figure \ref{4mfdsfig}(b). Then $B$ does not admit a Stein structure with either orientation.
\end{theorem}

This result does not obviously rule out a pseudoconvex embedding of $\Sigma(2,3,13)$ in $\cee^2$. However, as an application, we have the following result whose proof was communicated to us by Paul Melvin and Hannah Schwartz.

\begin{theorem}\label{corkthm} There exists a nontrivial cork that does not admit a Stein structure with either orientation.
\end{theorem}

Recall that a {\it (nontrivial) cork} is a pair $(C,g)$ where $C$ is a contractible 4-manifold and $g$ is a diffeomorphism of $\partial C$ that extends across $C$ as a homeomorphism but not as a diffeomorphism; see \cite{AKMR} for a recent discussion. Often $g$ is required to be an involution, and we arrange that here.

\begin{proof}[Proof of Theorem \ref{corkthm}] Let $(W,\tau)$ be any nontrivial cork with boundary involution $\tau$, for example we can take $W$ to be the Mazur manifold as in \cite{akbulut91}. Let $C$ be the 4-manifold obtained by the boundary sum of $W$ with two copies of the Akbulut-Kirby manifold $B$, where the two factors are summed at points exchanged by $\tau$. We can extend $\tau$ in an obvious way to an involution $\tilde\tau$ of $C$ exchanging the copies of $B$; we claim this is nontrivial. To see this, we can first embed $W$ in some 4-manifold $X$ such that a twist along $W$ by $\tau$ changes the smooth structure of $X$ (a closed manifold $X$ with this property is contained in \cite{akbulut91}, or see \cite[Section 9.3]{GS}). Since the contractible $B$ embeds in a 4-ball we can trivially find two disjoint copies of $B$ in $X-W$, and form the boundary sum using thickened arcs between each $B$ and $W$. Hence $C$ embeds in $X$, and replacing $C$ using the involution $\tilde\tau$ clearly has the same effect on $X$ as the original twist along $W$. This shows $C$ is a nontrivial cork.

To see $C$ is not Stein, we appeal to a result of Eliashberg \cite{eliashberg89} that implies that a 4-manifold given as a boundary sum admits a Stein structure if and only if each summand does; since $\pm B$ is not Stein neither is $\pm C$.
\end{proof}

\subsection*{Other remarks and examples.} 

With regard to Gompf's conjecture on pseudoconvex embeddings of Brieskorn homology spheres, our methods do not extend easily to other known families of Brieskorn spheres that embed smoothly in $\cee^2$ (the particular difficulty is the classification of tight contact structures on these manifolds). However, if we consider the broader class of Seifert fibered rational homology spheres, then classification results due to various authors \cite{T16,ghiggini08} rule out the existence of Stein structures on several infinite families of rational homology balls, at least with one orientation. For example, the Seifert manifold $M(-2;\frac{1}{p}, \frac{p-1}{p},\frac{p-1}{p})$ bounds a rational homology ball $X_p$ by work of Casson-Harer \cite{cassonharer}, and supports Stein fillable contact structures. Since none of these structures have $\theta=-2$, however, it follows that $X_p$ cannot be Stein with the given orientation (see \cite[Theorem 1.3]{ghiggini08} and \cite[Theorem 1.1]{Lisca3}). One can then obtain rational homology balls that are not Stein with either orientation, by forming a boundary connected sum of $X_p$ with $-X_p$ and appealing to Eliashberg's decomposition theorem for Stein fillings mentioned above \cite{eliashberg89}. However, if one is given a rational homology ball with irreducible boundary and wishes to rule out Stein structures with either orientation, one must understand contact structures on the boundary with either orientation. As pointed out to the authors by Luke Williams, the lens space $L(25,7)$ is the boundary of a rational homology ball (see Roberts \cite{roberts06}, for example), and no tight contact structures have $\theta = -2$ by an easy calculation using the classification of tight contact structures on lens spaces due to Honda \cite{H000} and Giroux \cite{girouxinvent}. Since $L(25,7) \cong -L(25,7)$ the same is true with the opposite orientation.

\subsection*{Acknowledgements} The authors wish to thank Bob Gompf for reminding us of his conjecture in \cite{gompf13}, Paul Melvin and Hannah Schwartz for pointing out Theorem \ref{corkthm} and its proof, and Alex Moody and Luke Williams for interesting conversations. We especially thank {\c{C}a\`{g}r\i} Karakurt for helping us find an error in an earlier version of this paper. The first author was supported in part by NSF grant DMS-1309212 and a grant from the Simons Foundation (523795, TM). The second author was supported in part by an AMS-Simons travel grant.

\section{Contact structures on $\Sigma(2,3,6m+1)$} \label{contactsec}

Here we describe the arguments leading to the classification result Theorem \ref{thm3} for contact structures on the Brieskorn manifolds we are considering. We consider the standard orientation of $\Sigma(2,3,6m+1)$, i.e., the manifold is oriented as the link of the complex surface singularity $x^2 + y^3 + z^{6m+1} =0$. The link is a Seifert fibered space over the 2-sphere, with three singular fibers having multiplicities 2, 3, and $6m+1$. From this one can obtain a surgery description for $\Sigma(2,3,6m+1)$ as follows. First one determines integers $r,s,t$ such that 
\[
3(6m+1)r + 2(6m+1)s + 2\cdot 3t = 1,
\]
for example $r=1$, $s = -1$ and $t = -m$. Then $\Sigma(2,3,6m+1)$ is described by performing surgery on a zero-framed unknot in $S^3$, as well as on three meridians with coefficients $2/r$, $3/s$, and $(6m+1)/t$ (see Figure \ref{anothersur}). 

\begin{figure}[h!]
\begin{center}
 \includegraphics[width=11cm]{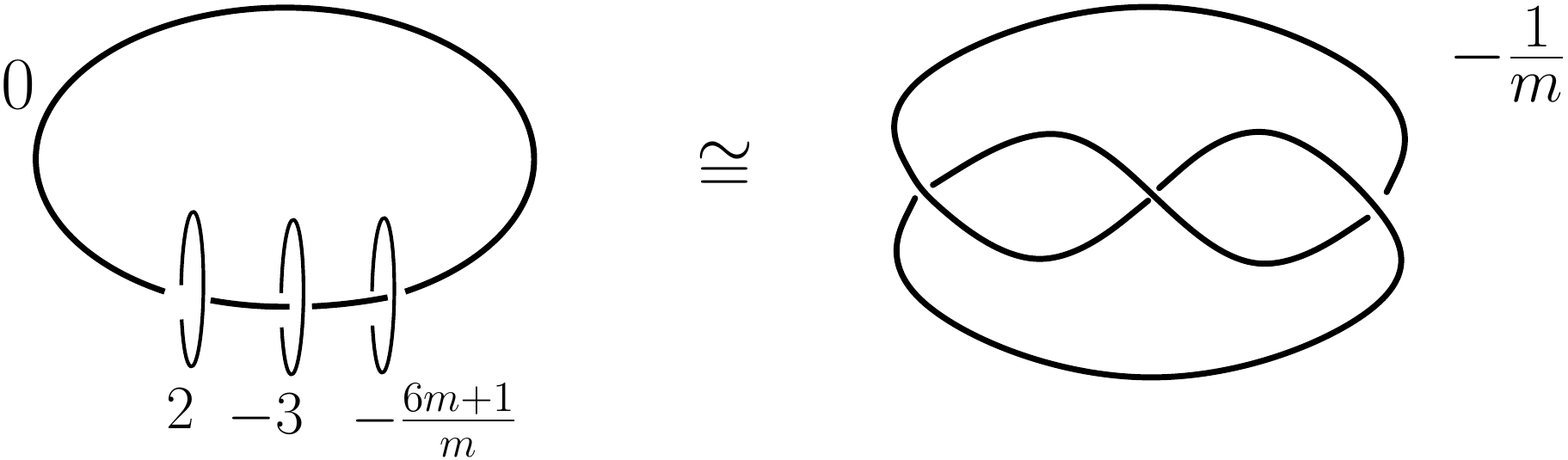}
 \caption{Surgery description of $\Sigma(2,3,6m+1)$.}
  \label{anothersur}
\end{center}
\end{figure}

\begin{lemma}\label{class1}
For $m\geq 1$, the manifold $\Sigma(2,3,6m+1)$  supports two non-isotopic, Stein fillable contact structures that are contactomorphic and have $\theta = -2$.
\end{lemma}  

\begin{proof} Perform a left Rolfsen twist on the 2-framed meridian in Figure \ref{anothersur}, so the coefficient on that component becomes $-2$ while the 0-framed circle now gets framing $-1$. Now we can perform three successive blowdowns (handleslides and deletions of $-1$ framed unknots) to obtain the right hand side of Figure \ref{anothersur}, which exhibits $\Sigma(2,3,6m+1)$ as the result of $-1/m$ surgery on the right trefoil. To produce a Stein manifold with boundary $\Sigma(2,3,6m+1)$, we realize the right trefoil in the standard way as a Legendrian in $S^3$ having Thurston-Bennequin invariant $+1$, so that the desired framing is $-\frac{m+1}{m}$ with respect to the contact framing. There is a standard way to describe a rational contact surgery by a contact surgery along a Legendrian link with all framings $\pm 1$ due to Ding-Geiges-Stipsicz \cite{DGS}; in this case the algorithm describes $-\frac{m+1}{m}$ surgery as Legendrian surgery on the link given by $m$ parallel copies of a once-stabilized right trefoil (see Figure \ref{contact}). The two diagrams on the right of Figure \ref{contact} arise from the two choices for the stabilization; there is an obvious contactomorphism relating the resulting contact structures, induced by revolution around a vertical axis. 
\begin{figure}[h!]
\begin{center}
 \includegraphics[width=11cm]{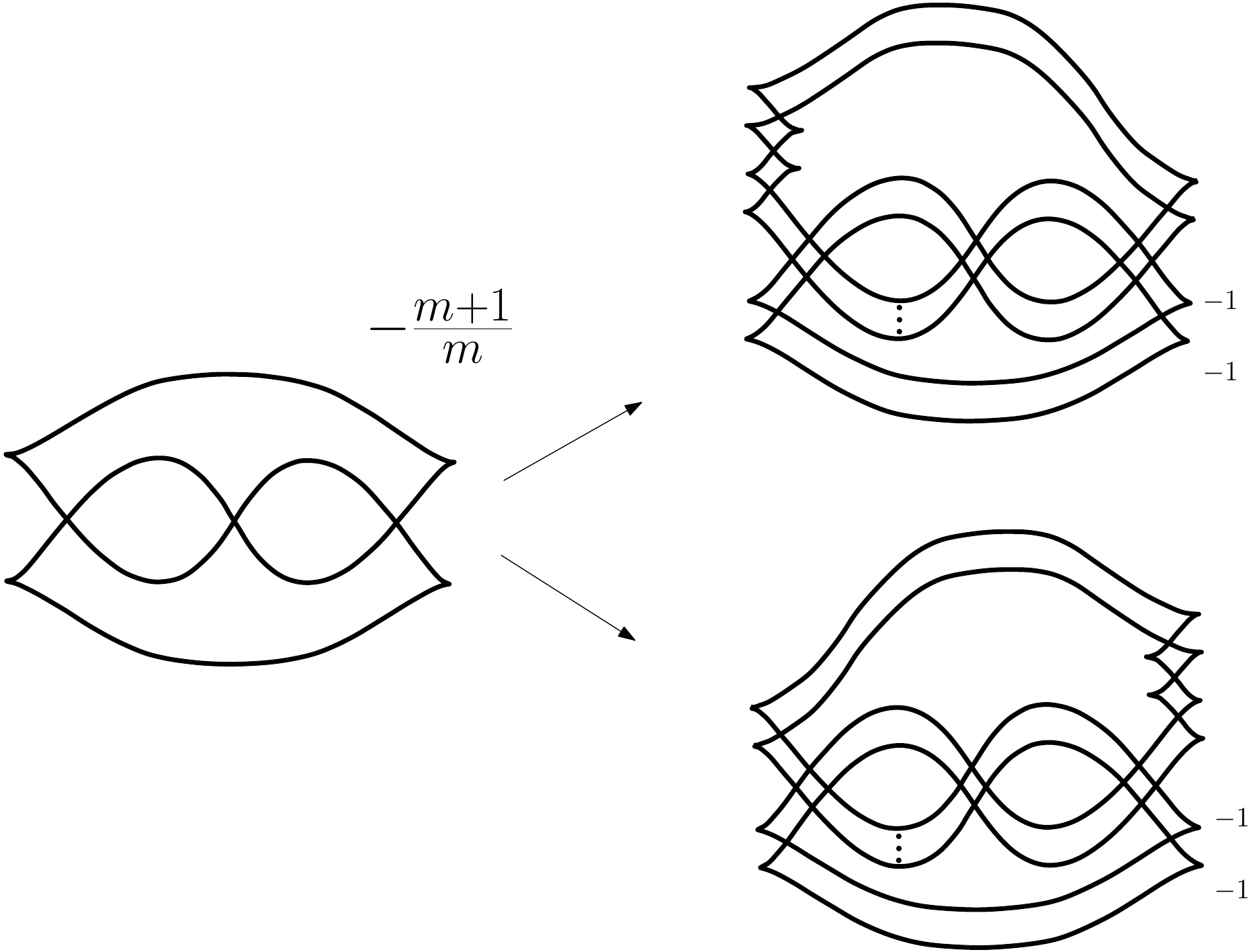}
 \caption{Two contactomorphic but not isotopic Stein fillable contact structures on $\Sigma(2,3,6m+1)$. The surgery coefficients are measured with respect to the contact framing; in each diagram on the right there are $m$ parallel pushoffs of the stabilized trefoil.}
  \label{contact}
\end{center}
\end{figure}

Since all surgeries have contact coefficient $-1$, the same diagrams describe Stein 2-handlebodies filling contact structures on $\Sigma(2,3,6m+1)$, and the Chern classes of the corresponding Stein structures evaluate on the handles corresponding to the Legendrians in the diagram as the rotation number (after fixing orientations for the Legendrians). The two Chern classes obviously differ by a sign, and by \cite[Theorem 1.2]{LM97} must induce non-isotopic contact structures on the boundary.

One easily calculates that for each of the two Stein structures just described the Chern class has square $-m$, and the claimed value for $\theta$ is then immediate.
\end{proof}

Having exhibited two Stein fillable (and in particular, tight) contact structures on $\Sigma(2,3,6m+1)$, the next result completes the proof of Theorem \ref{thm3}.

\begin{proposition} The Brieskorn manifold $\Sigma(2,3,6m+1)$ admits at most two tight contact structures.
\end{proposition}
\begin{proof}
The proof is presented assuming the reader is familiar with Giroux's convex surface theory \cite{G91} and Honda's  bypass technology \cite{H000}. The argument parallels the one carried out in Ghiggini-Sch\"onenberger \cite{paolostef}, and the reader is referred to that work for an excellent exposition.

 We first specify our framing convention. The manifold $M=\Sigma(2,3,6m+1)$ is, as explained above, a Seifert fibered space over $S^2$ with three singular fibers $F_i$, $i=1,2,3$.  Let $V_i \cong D^2\times S^1$ denote a solid torus neighborhood of $F_i$, and identify $\partial V_i\cong \mathbb{R}^2/\mathbb{Z}^2$ such that $(1,0)^T$ is the direction of the meridian.  We have $M\setminus V_{1}\cup V_{2}\cup V_{3}\cong \Sigma \times S^1$ where $\Sigma$ is a pair of pants, and we choose an identification $-\partial(M \setminus V_{i})\cong \mathbb{R}^2/\mathbb{Z}^2$ such that $(0,1)^{T}$ is the direction of the $S^1$ fiber and $(1,0)^{T}$ is the direction given by $-\partial({pt.} \times \Sigma)$. We then obtain $M$ as $M\cong(\Sigma \times S^1)\cup_{(A_{1}\cup A_{2}\cup A_{3})} (V_{1}\cup V_{2}\cup V_{3})$ where the attaching maps $A_{i}=\partial V_{i}\rightarrow -\partial(\Sigma\times S^1)_{i}$ are given by

\[
A_1=\left( \begin{array}{cc}
2&-1\\ 1&0
\end{array} \right), \qquad A_2=\left( \begin{array}{cc}
3 & 1 \\-1 & 0
\end{array} \right), \qquad A_3=\left( \begin{array}{cc}
6m+1 &  6m-5\\ -m & -m+1 
\end{array} \right).
\]
 
If $M$ is equipped with a contact structure, after isotopy  we can make each singular fiber $F_i$ Legendrian and take $V_{i}$ to be its standard neighborhood, with $\textrm{slope}(\Gamma_{\partial V_{i}})=\frac{1}{n_i}$ for some $n_i<0$. Here $\Gamma_{\partial V_i}$ indicates the dividing set on the torus $\partial V_i$, after a perturbation to make $\partial V_i$ convex, and ``slope'' refers to the trivialization described above. Thus the slope $1/n_i$ on $\partial V_i$ corresponds to the vector $(n_i, 1)^T$.  By Giroux's flexibility theorem \cite{G91}, we can isotope $V_i$ so that the ruling curves on $-\partial (M\setminus V_i)$ have infinite slope: we will call such curves {\it vertical} curves. (Note that the trivialization to be used on the torus is implicit in the notation $-\partial(M\setminus V_i)$, instead of $\partial V_i$.) In the following we often consider annuli with Legendrian boundary along the vertical curves on $-\partial(M\setminus V_i)$ and $-\partial(M\setminus V_j)$ for $i\neq j=1,2,3$. Such an annulus will also be called vertical.

When measured in $-\partial(M\setminus V_i)$ (that is, after applying the maps $A_i$ to the vectors $(n_i, 1)^T$), the slopes corresponding to $1/n_i$ are
\[
s_1=\frac{n_1}{2n_{1}-1},~ s_2=-\frac{n_{2}}{3n_{2}+1},~~s_3=-\frac{mn_{3}+m-1}{(6m+1)n_{3}+6m-5}.
\]

We now explain that by finding enough bypasses we can thicken $V_i$ such that the twisting numbers $n_i$ of the singular fibers can be increased to $n_1=n_2=-2$ and $n_3=0$. Let $\mathcal{A}$ be a vertical annulus between $V_1$ and $V_2$,  which we can assume is convex. There are two cases to analyze based on the slopes of tori that $\mathcal{A}$ connects and the dividing set configuration of $\mathcal{A}$. Observe that the number of endpoints of $\Gamma_\A$ on the boundary corresponding to $V_i$ is equal to the number of intersections between the relevant boundary of $\A$ with the dividing set $\Gamma_{-\partial (M-V_i)}$. Since the slope of $\partial \A$ is infinite and the dividing set on each torus contains two components, this intersection is equal to twice the denominator of the corresponding slope $s_i$.

\vspace{2mm}

{\it Case 1:} If $2n_{1}-1\neq 3n_{2}+1$, then the ``imbalance principle'' applies: the dividing set of $\mathcal{A}$ has at least one boundary parallel arc, which bounds a bypass disc. This bypass might allow an increase of the corresponding twisting number $n_1$ or $n_2$. More precisely, since the ruling slopes on $\partial V_1$ and $\partial V_2$ are $2$ and $-3$, respectively (obtained by transferring the infinite ruling slope on $-\partial (M \setminus V_i)$ using the inverse of $A_i$), the Twist Number Lemma from \cite[Lemma~$4.4$]{H000} says that a bypass found in this way allows us to increase $n_1$ and $n_2$ incrementally up to a maximum of $n_1=0$ and $n_2=-1$ (so long as we remain in Case 1).

\vspace{2mm}

{\it Case 2:} If $2n_{1}-1=3n_{2}+1$ and the dividing set of $A$ does not have any boundary parallel arc, then the dividing curves on $\mathcal{A}$ run across from $-\partial(M\setminus V_1)$ to $-\partial(M\setminus V_2)$. We can cut along $\mathcal{A}$ and round the corners to get a smooth manifold $M\setminus (V_1\cup V_2\cup \mathcal{A})$ such that $\partial(M\setminus (V_1\cup V_2\cup \mathcal{A}))$ is smoothly isotopic to $\partial(M\setminus V_3)$. Moreover, by the Edge Rounding Lemma from \cite[Lemma~$3.1$]{H000} we compute its slope as

\[
s(\Gamma_{\partial(M\setminus V_1 \cup V_2 \cup \mathcal{A})})=\frac{n_1}{2n_{1}-1} - \frac{n_{2}}{3n_{2}+1} - \frac{1}{2n_{1}-1}=\frac{n_{1}-1}{6n_{1}-3}.
\]    
 
The corresponding slope on $\partial V_3$ is obtained by first reversing the sign (to account for orientation), and applying the gluing map. We get

\[
s(\Gamma_{\partial V_3}) = A^{-1}_{3}(6n_{1}-3,-n_{1}+1)^T=\frac{-n_1+3m+1}{n_1-3m+2}.
\]

This slope is less than $-1$ for all $n_1\leq 0$ and $m\geq 1$. So, by \cite[Theorem $4.16$]{H000} (c.f. \cite[Lemma 3.16]{etnyrehonda2000}), for any negative $n_1$ we can find a convex neighborhood $V'_{3}\subset V_{3}$ of the singular fiber $F_3$ such that $s(\Gamma_{\partial V'_{3}})=-1$. When measured with respect to $-\partial(M\setminus V'_3)$, the slope becomes $-\frac{1}{6}$ which corresponds to $n_3=-1$. We now take a vertical annulus between $V_{1}$ and $V'_{3}$, and (comparing slope denominators) observe that $|2n_{1}-1|>6$ as long as $n_{1}<-2$. Therefore by the imbalance principle, the vertical annulus will have a bypass on $V_{1}$ side, and by the Twist Number Lemma we can repeatedly attach bypasses to $V_1$ to increase the twisting $n_{1}$ to $-2$. Similarly we can increase the twisting $n_{2}$ to $-2$. So, slopes in the coordinates of $-\partial(M\setminus V_i)$ become $s_1=\frac{2}{5}, s_{2}=-\frac{2}{5}$ and $s_{3}=-\frac{1}{6}$. 

We claim that a vertical convex annulus between $V_1$ and $V_2$ will now have no boundary parallel arcs in its dividing set, if the contact structure under consideration is tight. To see this, suppose there is a bypass on one side or the other: then since each boundary of $\A$ contains the same number of endpoints of $\Gamma_\A$  there must be a bypass on each side. Attaching these to $V_1$ and $V_2$ results in thickened convex neighborhoods that we denote with the same symbol, but which now have slopes $s_1 = s(\Gamma_{-\partial(M-V_1)}) = \frac{1}{3}$ and $s_2 = s(\Gamma_{-\partial(M-V_1)}) = -\frac{1}{2}$ (c.f. the ``bypass attachment lemma,'' \cite[Theorem 2.4]{paolostef} or \cite[Lemma 3.15]{H000}). Since the denominator of the first slope is greater than that of the second there is necessarily another bypass on the $V_1$ side; attach that to find a further thickening with slope $s_1 = 0$. Now the denominator of $s_2$ is larger, so we can attach a bypass to $V_2$ and find a thickening with slope $s_2 = -1$. At this point the denominators agree again, so either there are no further bypasses or there is at least one on each side. In the latter case we can thicken again to get $s_1=\infty$ (and $s_2=\infty$). In particular a vertical curve (i.e., a Legendrian regular fiber) on $-\partial(M\setminus V_1)$ will have twisting $0$. But by \cite[Lemma $4.11$]{paolostef} this is impossible as the maximal twisting number of any tight contact structure on a Seifert manifold $M(-\frac{1}{2}, \frac{1}{3}, r)$ is negative, for any $r\leq \frac{1}{5}$ (here the Seifert invariants $-\frac{1}{2}, \frac{1}{3}, r$ are the negative reciprocals of the surgery coefficients on the meridians in the diagram of Figure \ref{anothersur}). On the other hand, if the vertical annulus now does not carry any bypasses and $s_1=0$, $s_2=-1$ then we can cut along the vertical annulus and round the edges to get a torus with slope $0$. When measured with respect to $\partial V_3$, this slopes becomes $-\frac{m}{m-1}$. This negative quantity is less than $-\frac{6m+1}{6m-5}$ for any $m\geq 1$. Hence there exists a convex torus $V'_3$ in $V_3$ with slope $-\frac{6m+1}{6m-5}$, which becomes $\infty$ when measured with respect to $-\partial(M\setminus V'_3)$, yielding the same contradiction that there is a vertical Legendrian curve with zero twisting number.

So a vertical annulus between $V_1$ and $V_2$ cannot have boundary parallel arcs in its dividing set, and without loss of generality we can then assume $\Gamma_\A$ consists of horizontal arcs. Once again, we cut along $\mathcal A$ and round the corners of $(M\setminus V_1 \cup V_2 \cup\mathcal A)$. Then as before $\partial(M\setminus V_1 \cup V_2 \cup \mathcal A)$ is smoothly isotopic to $\partial(M\setminus V_3)$ and by the Edge Rounding Lemma has slope $\frac{2}{5}-\frac{2}{5}-\frac{1}{5}=-\frac{1}{5}$. Recall $s_1=\frac{2}{5}$ and $s_2=-\frac{2}{5}$ correspond to slopes $\frac{1}{n_1}=-\frac{1}{2}$ and $\frac{1}{n_2}=-\frac{1}{2}$, respectively. In particular $V_1$ and $V_2$ are the standard neighborhoods of Legendrian representatives of the corresponding singular fibers, and hence each carries unique tight contact structure up to isotopy (convex solid tori with dividing set consisting of two curves of slope equal to the reciprocal of a negative integer carry a unique tight contact structure; for example see \cite[Theorem 6.7]{etnyreintrolectures}). On the other hand, the slope $s_3 = -\frac{1}{5}$ corresponds on $\partial V_3$ to slope $-\frac{m+1}{m} = [-2,\ldots, -2]$, where the notation indicates the continued fraction expansion of the slope and includes $m$ copies of $-2$. It follows from the classification of tight contact structures on solid tori \cite[Theorem 2.3]{H000} that $V_3$ admits exactly two tight contact structures extending this boundary condition. We can decompose $M$ as $(M\setminus V_3)\cup V_3$ where $M\setminus V_3$ is made of $V_1$, $V_2$ and a neighborhood of the annulus $\mathcal A$. As the dividing set of $\mathcal A$ uniquely determines a tight contact structure in the neighborhood of $\A$, and $V_1$, $V_2$ each carries a unique tight contact structure, we infer that $M\setminus V_3$ has unique tight contact structure relative to its boundary $\partial (M\setminus V_3)=\partial V_3$. On the other hand,  as explained above $V_3$ has two tight contact structures satisfying the same boundary condition. So, $M$ has at most two tight contact structures.


\end{proof}

\section{Proof of Theorem \ref{AKthm}}\label{HFsec}

In this section we prove that the particular contractible 4-manifold $B$ with $\partial B = \Sigma(2,3,13)$ constructed by Akbulut and Kirby in \cite{AKmazur} does not admit a Stein structure with either orientation. Our argument is specific to this 4-manifold, but relies on constructions that we phrase generally for as long as possible. Note that $-B$ admits no Stein structure by Theorem \ref{mainthm}; hence to prove Theorem \ref{AKthm} it suffices to show $B$ has no Stein structure either.

First, recall that the Milnor fiber of the Brieskorn singularity $z_1^2 + z_2^3 + z_3^{6m+1} = 0$ admits a natural (smooth) compactification whose divisor at infinity is described topologically by the plumbing graph of Figure \ref{capplumbingfig} (cf. Ebeling-Okonek \cite[Section 3]{EO}). Let $P_m$ be a neighborhood of this divisor and $M_m$ the Milnor fiber, so the compactification is a smooth complex surface $Z_m$ diffeomorphic to $M_m\cup_\partial P_m$ (in fact, up to diffeomorphism $Z_m$ is nothing but the elliptic surface $E(m+1)$). Observe that $P_m$ contains a neighborhood of an $\widetilde{E}_8$ singular fiber in an elliptic fibration (the union of the $-2$ spheres), and in particular contains a complex curve $F$ of genus 1 and self-intersection 0 corresponding to a smooth fiber of this fibration. It follows from the adjunction formula or the description of $Z_m$ as an elliptic surface that the canonical class of $Z_m$ is Poincar\'e dual to the class $(m-1)[F]$.

\begin{figure}[t]
\def\svgwidth{4in}
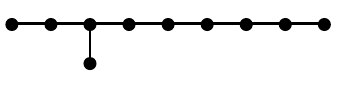
\caption{\label{capplumbingfig}}
\end{figure}

\begin{lemma}\label{nonrationallemma} Suppose $B_m$ is an acyclic Stein manifold with boundary $\partial B_m = \Sigma(2,3,6m+1)$. Then the smooth 4-manifold $Y_m = B_m\cup_\partial P_m$ obtained from $Z_m$ by replacing the Milnor fiber with $B_m$ admits a symplectic structure whose canonical class is Poincar\'e dual to (the image in $Y_m$ of) the class $(m-1)[F]$. If $m$ is even and at least 2, then $Y_m$ is homeomorphic but not diffeomorphic to the rational elliptic surface $E(1)$.
\end{lemma}

Strictly, the smooth type of $Y_m$ may depend on the choice of diffeomorphism $\partial B_m\to -\partial P_m$; the lemma is to be understood as asserting the existence of a diffeomorphism resulting in a manifold $Y_m$ with the stated properties. 

\begin{proof} Essentially by construction, we can arrange that $M_m\subset Z_m$ is symplectically embedded with convex boundary. Correspondingly, $\partial P_m$ is concave with respect to a suitable symplectic structure on $P_m$ (coming from the K\"ahler structure on $Z_m$; this also follows from general results such as \cite[Theorem 1.3]{li-mak}). By Theorem \ref{thm3}, there is a contactomorphism $\partial B_m\to -\partial P_m$ that we can use to obtain a symplectic structure on $Y_m$. Since the canonical class of $Z_m$ is supported in $P_m$ it is clear that $Y_m$ has the stated canonical class. 

For the homeomorphism classification of $Y_m$, observe that since $B_m$ is Stein it admits a handle decomposition without 3-handles. Turning this handle decomposition over, $B_m$ can be built from $\partial B_m$ by handles of index $\geq 2$, and in particular the fundamental group of $\partial B_m$ surjects onto that of $B_m$. It follows easily that $Y_m$ is simply connected; since $m$ is even the intersection form of $Y_m$ is odd with $b^+ =1$ and $b^- = 9$ (in fact the intersection form of $Y_m$ is the same as that of the plumbing $P_m$ in Figure \ref{capplumbingfig}). Hence Freedman's theorem implies $Y_m$ is homeomorphic to $E(1)$.

To distinguish the diffeomorphism type of $Y_m$ from that of $E(1)$, we turn to Seiberg-Witten invariants. Recall that given a Riemannian 4-manifold $Y$ with $b^+(Y) = 1$ and fixed orientation of the space $\mathcal{H}^+$ of self-dual harmonic 2-forms, a \spinc structure $\s$ has two Seiberg-Witten invariants $SW^+(Y,\s)$ and $SW^-(Y,\s)$ (that are independent of the Riemannian metric) and also a ``small-perturbation invariant'' $SW^0(Y,\s)$ given by
\[
SW^0(Y,\s) = \left\{\begin{array}{ll} SW^+(Y,\s) & \mbox{if $c_1(\s)\cup [\omega] >0$} \\
SW^-(Y,\s) & \mbox{if $c_1(\s)\cup[\omega] <0$}.\end{array}\right.
\]
Here $[\omega]$ is the cohomology class of a closed, nonzero, self-dual 2-form with respect to the metric on $Y$, which determines the given orientation of $\mathcal{H}^+$. We are following the conventions of Li and Liu \cite{liliuuniqueness}; see also Park \cite{park05} or Fintushel-Stern \cite{FSsixlectures} for expositions. (Note also that the case $c_1(\s)\cup[\omega]=0$ does not arise for \spinc structures having Seiberg-Witten moduli space of nonnegative dimension.)  When $Y$ is simply connected with $b^-(Y)\leq 9$, it follows from the light cone lemma that $SW^0(Y,\s)$ does not depend on the choice of metric used in its definition and is therefore a diffeomorphism invariant.

If $Y$ is symplectic, we can take $[\omega]$ to be the cohomology class of the symplectic form. Moreover, if $\s_\omega$ is the canonical \spinc structure on $(Y,\omega)$, having first Chern class $c_1(\s_\omega) = c_1(Y)$, then work of Taubes \cite{taubessymp} shows that $SW^-(Y, \s_\omega) = 1$. On the other hand, if $Y$ admits a metric of positive scalar curvature, in particular if $Y = E(1)$, then the small-perturbation invariants of $Y$ all vanish.

Turning to $Y_m$, observe that $c_1(\s_\omega)$ is Poincar\'e dual to $(1-m)[F]$, and hence $c_1(\s_\omega)\cup[\omega] = (1-m) \int_F \omega < 0$ since $F$ is a symplectic torus. By Taubes' result and the definition of $SW^0$, we infer $SW^0(Y_m, \s_\omega) = SW^-(Y_m, \s_\omega) = 1$, which shows $Y_m$ is not diffeomorphic to $E(1)$.
\end{proof}

Recall that the boundary of the neighborhood of an $\widetilde{E}_8$ fiber is diffeomorphic to the result of 0-framed surgery on the left-hand trefoil knot. The diffeomorphism can be realized by blowing up the $-2$ framed vertex at the end of the long leg of the $\widetilde{E}_8$ graph using a $+1$ circle, then sequentially blowing down $-1$ circles. Performing this construction on the $\widetilde{E}_8$ graph embedded in $P_m$, beginning by blowing up between the $-2$ and $-(m+1)$ framed vertices using a $+1$ curve and carrying along the resulting $-m$ circle through the process, proves:

\begin{lemma}\label{cobordlemma} The cobordism $\partial\widetilde{E}_8\to \partial P_m$ given by attaching the handle $h$ corresponding to the $-(m+1)$ framed vertex in Figure \ref{capplumbingfig} is described by the diagram in Figure \ref{4mfdsfig}(a).\hfill$\Box$
\end{lemma} 

\begin{figure}[b]
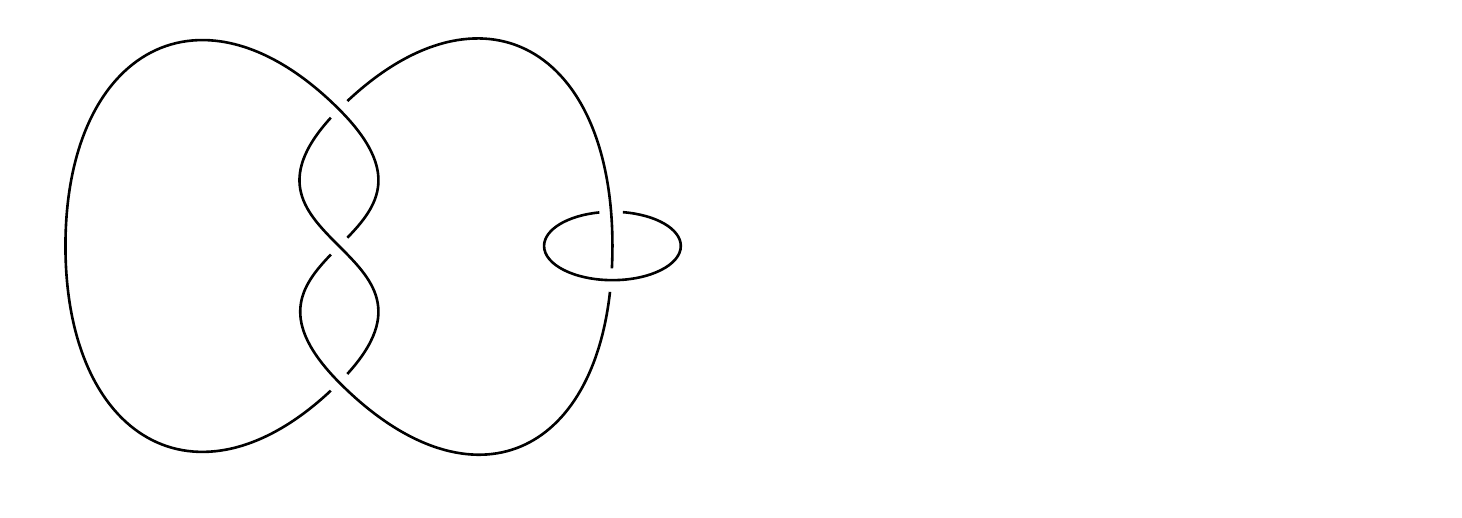
\begin{caption}{\label{4mfdsfig} On the left, the cobordism $\partial \wt{E}_8\to \partial P_m = -\Sigma(2,3,6m+1)$ given by the handle corresponding to the rightmost vertex in Figure \ref{capplumbingfig}. On the right, the 4-manifold $B_2$ with $\partial B_2 = \Sigma(2,3,13)$ described by Akbulut and Kirby.}\end{caption}
\end{figure}

In Figures \ref{4mfdsfig}(a) and \ref{sept12fig} we follow the conventions of Gompf and Stipsicz \cite{GS}, whereby framings without brackets refer to 2-handles of a cobordism built on a 3-manifold that may not be $S^3$, which is described by a surgery diagram with coefficients in brackets.

Turning the cobordism of Lemma \ref{cobordlemma} upside-down, we obtain a cobordism
\[
W_m: \Sigma(2,3,6m+1) = -\partial P_m \to -\partial \wt{E}_8,
\]
described by adding a 0-framed meridian to the $-m$ framed curve in Figure \ref{4mfdsfig}(a) and mirroring the diagram to obtain Figure \ref{sept12fig}(a) (in the case $m = 2$).

At this point we specialize to the case $m = 2$ and consider the Akbulut-Kirby contractible 4-manifold with boundary $\Sigma(2,3,13)$, which is given by the handle diagram in Figure \ref{4mfdsfig}(b).

\begin{lemma}\label{cusplemma} Let $B_2$ denote the contractible 4-manifold of Figure \ref{4mfdsfig}(b). Then the union $B_2 \cup W_2$ is diffeomorphic to the neighborhood of a cusp fiber in an elliptic fibration. 
\end{lemma}

Recall that a cusp neighborhood is described by a handle picture with a single 0-framed 2-handle attached along a right trefoil knot; see \cite[Figure 8.9]{GS}.

\begin{figure}
\begin{centering}
\hspace*{1cm}\parbox{6in}{\def\svgwidth{6in}
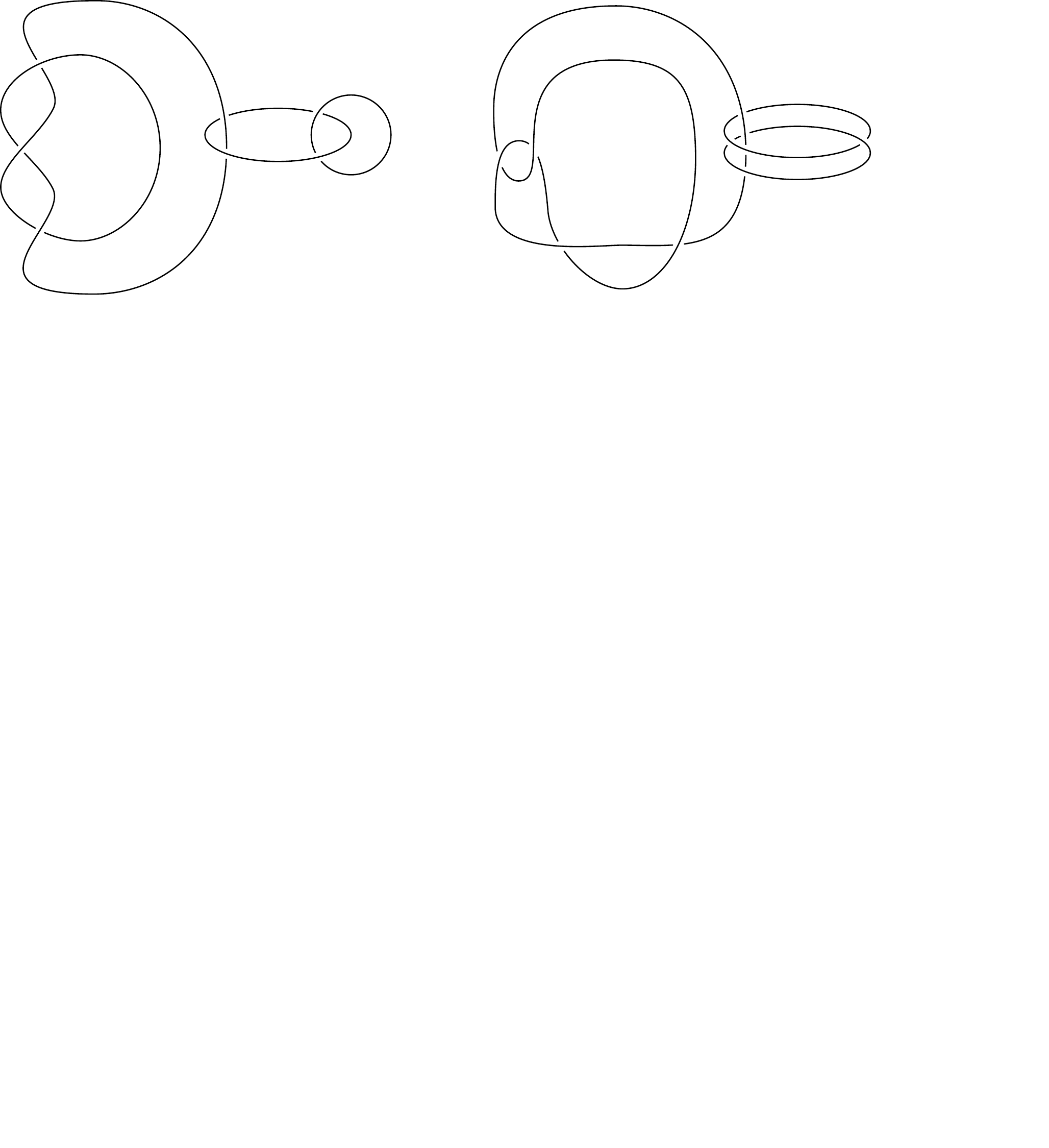}
\end{centering}
\caption{\label{sept12fig} Steps in the proof of Lemma \ref{cusplemma} (essentially mimicking the procedure in \cite{AKmazur}). From (a) to (b) apply a left Rolfsen twist around the 2-framed circle and isotop. Then blow up the clasp with a $+1$ circle and further isotop to obtain (c); two blowdowns of $-1$ curves give (d). Diagram (e) is obtained by isotopy and (f) results from blowing up the full twist in (e). }
\end{figure}

\begin{proof} The first diagram of Figure \ref{sept12fig} shows the cobordism $W_2$, where the trefoil and its meridian give a surgery description of $\Sigma(2,3,13)$ and the 0-framed unknot is the single 2-handle in $W_2$. Our strategy is to exhibit a diffeomorphism between this surgery picture for $\Sigma(2,3,13)$ and the one arising at the boundary of $B_2$ in Figure \ref{4mfdsfig}(b), carrying along the 0-framed knot; this is done in the sequence of diagrams in Figure \ref{sept12fig}. In diagram (f) in that figure the bracketed surgery diagram describes the boundary of $B_2$: indeed, the bracketed two-component link in that diagram is symmetric. Hence, placing a dot on the bracketed 0-framed circle of Figure \ref{sept12fig}(f), and removing the brackets from the $-1$ circle, gives a Kirby picture for $B_2\cup W_2$.

However, it is now clear that by sliding the $-1$ circle over the 0-framed handle representing $W_2$ in Figure \ref{sept12fig}(f), we can cancel the 1-handle in $B_2\cup W_2$. Carrying the remaining 2-handle through this cancellation gives the description of $B_2\cup W_2$ as obtained by adding a single 0-framed 2-handle to $B^4$ along a right trefoil knot. (We leave this as an exercise for the reader, but in fact since $\partial(B_2\cup W_2)$ is diffeomorphic to the result of $0$-surgery on the right trefoil, and we now know $B_2\cup W_2$ has a description consisting of a single 2-handle, it is a consequence of \cite[Corollary 8.19]{gabaifol3}---which asserts that a knot is fibered if and only if 0-surgery on the knot is fibered---that the handle is necessarily attached along the right trefoil.)

Strictly, what we have shown is that there is a diffeomorphism between $\partial B_2$ and $-\partial_-W_2$ such that the resulting glued manifold is a cusp neighborhood. However, the conclusion holds regardless of the choice of diffeomorphism, which we can see as follows. Recall that $\partial B_2 = \Sigma(2,3,13)$ is a Seifert manifold over $S^2$ with three exceptional fibers of orders 2, 3, and 13. Up to isotopy, $\Sigma(2,3,13)$ admits only one non-identity diffeomorphism $\phi$; we claim that the attaching circle for the 2-handle in $W_2$ is invariant under $\phi$, and from this the lemma follows. 

To check the invariance, we transform the first diagram of Figure \ref{sept12fig} into a standard surgery picture for a Seifert manifold, in which $\phi$ is easily described. The steps are carried out in Figure \ref{oct24fig}, in the last diagram of which is the surgery picture for $\Sigma(2,3,6m+1)$ viewed as the Seifert manifold $M(-1;\frac{1}{2},\frac{1}{3},\frac{m}{6m+1})$ (obtained from Figure \ref{anothersur} by a twist around the 2-framed circle), with the attaching circle for the handle indicated. The diffeomorphism $\phi$ is given in this picture by revolving the diagram around a horizontal line in the plane of the page by half a revolution, under which the attaching circle is clearly invariant. 
\end{proof}

\begin{figure}
\def\svgwidth{6in}
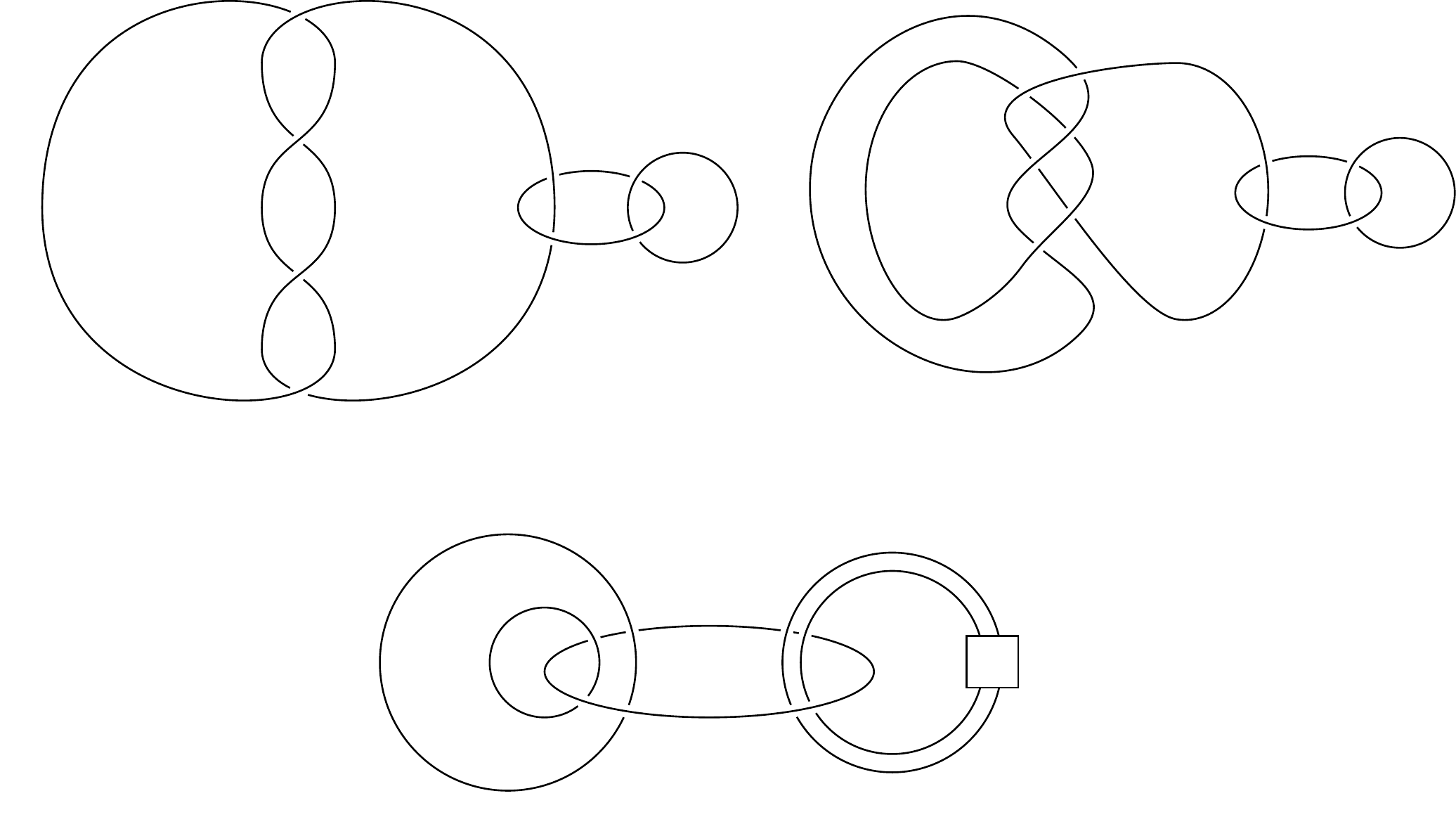
\caption{\label{oct24fig} Here curves labeled $h$ indicate the attaching circle for the cobordism $W_m$. In (a) and (b) $h$ carries framing 0 while in (c) the framing is $-6$. Diagram (a) is obtained from Figure \ref{sept12fig}(a) by blowing up a twist in the trefoil and an isotopy, along with replacing the framing 2 by the general $m$ (and ignoring brackets). Another blowup gives (b), and (c) follows by blowing up the full twist and a slam-dunk of the $m$-framed circle after sliding $h$ off it. Observe that a half-revolution of the diagram (c) around a horizontal axis preserves the diagram including the circle $h$.}
\end{figure}

\begin{theorem} The 4-manifold $Y_2 = P_2\cup B_2$, where $B_2$ is the Akbulut-Kirby manifold, is diffeomorphic to $E(1)$.
\end{theorem}

This result, combined with Lemma \ref{nonrationallemma}, shows that $B_2$ cannot admit a Stein structure and completes the proof of Theorem \ref{AKthm}.

\begin{proof}
It is well-known that $E(1)$ admits a decomposition $E(1) \cong \wt{E}_8\cup C$, where $C$ is a cusp neighborhood. We have seen above that 
\[
Y_2 = P_2 \cup B_2 = \wt{E}_8 \cup W_2 \cup B_2 = \wt{E}_8 \cup C,
\]
so the only issue is the identification between the boundaries of $\wt{E}_8$ and $C$. But the boundary of $C$, diffeomorphic to the result of 0-surgery on the right trefoil, is a Seifert manifold over a 2-sphere with three multiple fibers of orders 2, 3 and 6; similarly to $\Sigma(2,3,13)$ its mapping class group has only one nontrivial element corresponding to a half-revolution around a line in a suitably symmetric surgery diagram. In this case we can take this diagram to be just the 0-framed right trefoil knot (as it appears in Figure \ref{sept12fig}(a), for example) and the indicated diffeomorphism clearly extends over the 4-manifold $C$.
\end{proof}

\bibliographystyle{plain}
\bibliography{mybib}

\end{document}